\documentclass[11pt]{amsart}
\usepackage{epsfig,amsmath,amssymb,amsthm,bm}
\usepackage{enumerate}
\usepackage{amsfonts,amsmath}

\newcommand{\IZ}{\mathbb{Z}}                       
\newcommand{\IR}{\mathbb{R}}   

\newtheorem{stat}{Statement}
\newtheorem{decth}[stat]{Theorem}
\newtheorem{prop}[stat]{Proposition}
\newtheorem{cor}[stat]{Corollary}
\newtheorem{thm}[stat]{Theorem}
\newtheorem{lemma}[stat]{Lemma}
\newtheorem{definition}{Definition}
\theoremstyle{remark}
\newtheorem{remark}{\bf\em Remark}

\renewcommand{\epsilon}{\varepsilon}
\renewcommand{\phi}{\varphi}
\def\diam{\operatorname{diam}}
\usepackage[usenames]{color}
\usepackage[colorlinks=true]{hyperref}
\newcommand{\IP}{\mathbb{P}}                       
\newcommand{\IE}{\mathbb{E}}                       
\newcommand{\IB}{\mathbb{B}}                       
\newcommand{\I}{{\text{\Large $\mathfrak 1$}}}
\parskip        \smallskipamount

\usepackage{marginnote}

\begin{document}

\title[Power law condition for stability of Poisson hail]{Power law condition
for stability \\of Poisson hail}

\author{Sergey Foss}
\thanks{SF was supported by EPSRC grant EP/I017054/1}
\address{Sergey Foss:
School of Mathematical and Computer Sciences,
Heriot-Watt University,
Edinburgh, EH14 4AS, UK and Sobolev Institute of Mathematics and
Novosibirsk State University, Russia}
\email{S.Foss@hw.ac.uk}

\author{Takis Konstantopoulos}
\thanks{TK was supported by Swedish Research Council grant 2013-4688}
\address{Takis Konstantopoulos:
Department of Mathematics, Uppsala University,
P.O.\ Box 480, 751 06 Uppsala, Sweden}
\email{takiskonst@gmail.com}

\author{Thomas Mountford}
\address{Thomas Mountford:
Ecole Polytechnique F\'ed\'erale de Lausanne,
Institut de Math\'ematiques,
Station 8, 1015 Lausanne, Switzerland}
\email{thomas.mountford@epfl.ch}

\begin{abstract}
The Poisson hail model is a space-time
stochastic system introduced by Baccelli and Foss \cite{BF}
whose stability condition is non-obvious owing
to the fact that it is a spatially infinite.
Hailstones arrive at random points of time and
are placed in random positions of space.
Upon arrival, if not prevented by previously accumulated 
stones, a stone starts melting at unit rate. 
When the stone sizes have exponential tails then
stability conditions exist.
In this paper, we look at heavy tailed stone sizes
and prove that the system can be stabilized when the rate
of arrivals is sufficiently small. We also show that
the stability condition is, in a weak sense, optimal.
We use techniques and ideas from greedy lattice animals.


\end{abstract}

\maketitle

\section{Introduction}
The purpose of this article is to loosen conditions for stability in the
``Poisson hail" interacting queueing model introduced by \cite{BF}. In the discrete setting for this model, 
there are countably many jobs (identified by countably many points in space-time).
Job $i$ requires service $\tau_i $ from a subset $B_i \subset \IZ^d$.
As in the preceding paper,  we associate to a job $i$
a (semi-arbitrary) server $x = x(i) \in \IZ^d$ who is in some sense central in
the group $B_i$. We suppose that for each site $w \in \IZ^d$ the jobs $i $ with $x(i)
= w$ arrive according to a Poisson process $N_w$ of rate $\lambda $.  
The jobs $i$ arriving at $w$ will have their subsets $B_i $ and service times
$\tau_i $ distributed as i.i.d.\ vectors, so the arrivals at site $w$ may be
considered as a marked Poisson process $\Phi_w$. 
In other words, $\Phi_w$ is a Poisson process on 
$\IR \times \IR_+ \times 2^{\IZ^d}$. Points in its support
are typically denoted by $(t, \tau, B)$, the $t$'s forming the
aforementioned rate-$\lambda$ Poisson process. 
The pair $(\tau, B)$ is referred to as the mark of the point $t$.
We also assume that  the $\Phi_w$ are obtained as follows:
Let, for each $w$, $\widetilde \Phi_w$ be an independent copy of 
$\Phi_0$. Then let $\Phi_w$ contain all points of the form 
$(t, \tau, B+w)$ where $(t,\tau, B)$ is a point of $\widetilde \Phi_w$.
Thus, the arrival process (including marks) is translation invariant.
Physically, we can think of the system as a model of hailstones of 
cylindrical shape
$B \times [0,\tau] \subset \IZ^{d+1}$, where $\tau$ is the height of the stone
and $B$ its base. When a hailstone appears at some point of time at which
all sites $w \in B$ are free, it starts melting at rate $1$.
If there is at least one $w \in B$ occupied by a previously arrived
stone, then the current stone will not start melting before 
all sites in $w$ become free; at the first moment of time this happens,
the hailstone starts melting at rate $1$. (Only the ground, $\IZ^d$ is hot
and heat is not transmitted upwards!)
At each time $t$, we let $W(t,x)$ be the total work required
for $x$ to be free of hailstones provided no stones arrive after $t$.
In queueing terms, $W(t,x)$ is a workload. In hailstone terms, 
$W(t,x)$ is the sum of the heights of all hailstones which contain $x$ in
their base and have not been melted yet. Since the superposition
of $N_w$, $w \in \IZ^d$,  has infinite rate, it follows that within any
time interval of positive length there are infinitely many stones arriving. 
Thus $W(t,\cdot)$ will change infinitely many times in any right neighborhood of $t$.
However, typically, 
for fixed $x \in \IZ^d$, and any $\epsilon > 0$, $W(t,x)$ depends only
on $W(t-\epsilon, y)$, for $y$ ranging in a finite 
(but random) number of sites. This is
due to the fact that the only have to look at those $\Phi_w$ with points
$(s,\tau, B)$ such that $t-\epsilon \le s \le t$ and $x \in B$.

Fix $x \in \IZ^d$ and suppose there is $w \in \IZ^d$ 
such that $(t,\tau, B)$ is a point of $\Phi_w$.
Then 
\begin{equation}
\label{evol}
W(t+,x) = 
\begin{cases}
\max_{y \in B} W(t-,y) + \tau, & x \in B
\\
W(t-, x), & x \not \in B.
\end{cases}
\end{equation}
By convention, we shall assume that $t\mapsto W(t,x)$ is right-continuous: 
$W(t,x)$ $= W(t+,x)$.
On the other hand, if there is no $w$ such that $(t,\tau,B)$ is a point of 
$\Phi_w$ 
with $x \in B$ then $W(s,t)$, $s \ge t$, decreases linearly for a interval
of positive length until either it reaches zero or there is job arriving
at some $s > t$ at some site $w$ whose base contains $x$.
We have thus completely specified the dynamics of the system.
The system considered here differs from that of \cite{KB} in that
the latter (i) considers only finitely many sites ($\IZ^d$ is replaced by
a finite set) but (ii) works for stationary and ergodic arrival processes. 

The system is said to be stable if (starting from full vacancy at time $0$) the
distribution of $W(t,x)$ is tight as $t$ varies for fixed $x$. 
The central question to be addressed is when is the system stable
(for $\lambda $ sufficiently small).
More precisely, for which laws on $(\tau,B)$ for jobs arriving at the origin is
it the case that there exists $\lambda_0 \in (0, \infty ) $ so that 
the system is stable for all arrival rates $\lambda < \lambda_0$. 
To avoid trivialities, we assume that $B$ is a finite set, a.s. .
The founding article
\cite{BF} showed that the system was indeed stable provided that there is
$c \in (0, \infty ) $ so that
\[
\IE[e^{c (\tau + (\diam B)^d)}  ] <  \infty , 
\]
where $\diam B$ is the diameter of set $B$, i.e., the maximum
of $|x-y|_\infty$ over all $x, y \in B$, and where $|x|_\infty:=
\max_{1\le  i \le d} |x_i|$. The proof in \cite{BF} is based on
a comparison with an auxiliary branching process with weights
requiring the condition stated in the last display.

Our purpose in this paper is to slacken
this condition to the existence of the $(d+1+ \epsilon )$-th moment for $\tau +
\diam B$.
We then (easily) show that this condition is (in a certain weak sense) almost
optimal. The key idea is to use ideas on laws of large numbers for lattice
animals.
 This was first proved in \cite{CGGK}, though for this paper we take as
reference the article by James Martin \cite{JM}. 
In analogy to \cite{KB}, one could also ask whether stability is possible
for more general arrival processes. This question, however, is outside the scope
of our paper as our method explicitly uses the Poissonian assumptions.

Our principal result is
\begin{thm} \label{thm1}
Suppose there exists $\epsilon >0$ such that  $\tau$ and 
$\diam B$ have 
finite moments of order $(d+1+\varepsilon)$:
\[
\IE \tau^{d+1+\epsilon} + \IE (\diam B)^{d+1+\epsilon} <\infty.
\]
Then there exists $\lambda_0 > 0 $ so that for job arrival rate $ \lambda <
\lambda _0 $ the system is stable.
\end{thm}

That this result is (in a weak sense) the best possible is shown by
\begin{thm} \label{thm2}
For any $ d+1 > \epsilon > 0$, we can find a (spatially homogeneous) job arrival
process so that  
\[
\IE \tau^{d+1-\epsilon} + \IE (\diam B)^{d+1-\epsilon} <\infty
\]
and the system is unstable.
\end{thm}

\begin{remark}
The condition of Theorem \ref{thm1} is equivalent to the 
following:
there exists $\epsilon >0$ and $C>0$ such that, for any $x\ge 0$,
\begin{equation}\label{maincondition}
\IP(\tau + \diam B > x) \le
\frac{C}{x^{d+1+ \epsilon }}.
\end{equation}
Similarly, the condition of Theorem \ref{thm2} is equivalent to 
the same thing with $-\epsilon$ in place of $\epsilon$. 
\end{remark}

Given stability, it is easy to see that starting from complete vacancy (that is,
no workload at any site), the system converges in distribution to an explicitly
describable
equilibrium. It is natural to ask whether the system possesses other, 
not necessarily spatially homogeneous, equilibria.
While not definitively answering this we show
\begin{thm} \label{Thm3}
Under the conditions of Theorem \ref{thm1}, there exists $\lambda_0 > 0 $ so
that for arrival rate $0 < \lambda < \lambda_0 $, the only equilibrium for the
system that is
spatially translation invariant is the limit measure obtained by starting from
zero workload.
\end{thm}

We now assemble some observations and techniques from earlier papers,
\cite{KB,BF}. Start the system at time $-n$ from full vacancy and consider how
the workload $W^n(t,x)$ at time $t \ge -n$ and site $x \in \IZ^d$ is obtained.
\begin{definition}
\label{defgamma}
Let $\Gamma^n(x,t) $ be the set of locally constant cadlag (=
piecewise constant, continuous on the right with left limits at
every point)  paths $
\gamma : [u,t] \rightarrow \IZ^d $ for some $-n \le u \le t $ such that \\
(i) $ \gamma (t)  =  x$,\\
(ii) if $ \gamma (s) \ne \gamma (s-)$, then a job arrived at time $s$ requiring
service from both servers $ \gamma (s) $ and $ \gamma (s-)$.
\end{definition}
Associate to such a $\gamma \in \Gamma^n(x,t)$ 
the score 
\[
V( \gamma) = \sum_i  \tau_i   -  (t-u),
\]
where the sum is over jobs $(\tau_i, B_i )$ which are arrive at time $s_i $ with
$\gamma(s_i) \in B_i $.
Based on the way that the workload evolves (see discussion around 
equation \eqref{evol}) we obtain that $W^n(t,x) = \sup_{\gamma
\in \Gamma^n(x,t)} V(\gamma)$.
See Figure \ref{graphicalrep}.

There are three monotonicity properties that the system possesses and which
we take into account when analyzing its stability. We start from full vacancy 
at time $-n$ and consider $W^n(t,x)$ for some $t \ge -n$. Then $W^n(t,x)$
will increase if we (i) delay all arrivals between $-n$ and $t$, 
or (ii) increase
the heights of the stones, or (iii) enlarge their bases.

Thanks to the monotonicity, it was deduced in \cite{BF} that it is  enough, 
for the
results sought, to consider the case where the sets $B$ for the team of servers
required for a job $i$ with $x(i) = 0$ is a cube centred at the origin and we
write (for a job arriving at server $x$ in time interval $(m-1,m]$)
$R^{x,m}_i$ for the value so that $B  = x +  [-R^{x,m}_i,R^{x,m}_i]^d$. We
will work with time doubly infinite, notwithstanding the fact that we consider
the process on $[-n, \infty )$.

The first step is the discretization of the Poisson processes. We consider for
$m \in \IZ $ and $x \in \IZ^d$ the random variables
\[
R_{x,m}  =  \sum _{i } R^{x,m}_i, \quad T_{x,m}  =  \sum _{i} \tau^{x,m}_i ,
\]

the sum taken over all jobs $i$ arriving on the time interval
$(m-1,m]$ at site $x$. 
Since the summands in $R_{x,m}$ are i.i.d.\ and their number is Poisson
(and, therefore, light-tailed), it is clear that
$R_{x,m}$ has finite $\alpha$ moment if each summand has finite
$\alpha$ moment. Similarly for $T_{x,m}$.
\begin{lemma} \label{basic}
If for $ \alpha > 0$, $\IE (\diam B)^{\alpha} <\infty$
(resp., $\IE \tau_i^{\alpha} <\infty$),
then
$\IE R_{x,m}^{\alpha} <\infty$
(resp. $\IE T_{x,m}^{\alpha}<\infty$).
\end{lemma}
We will deal with the discretized system where at server $w$ at time $n$ a job
requiring service time $T_{w,n} $ from each server in the cube
$[w-R_{w,n},w+R_{w,n} ] ^d $.
By the monotonicity (see \cite{BF} for detail), this discretization is 
effective in that if we can
show stability for the discretized system of jobs then we will have shown
stability for the original system: the workload at time $m$ for this system will
dominate that arising from the nondscretized model. It is also as well to note
that we have not given up too much here. In principle, if we have multiple
$w$ jobs arriving during interval $(m-1,m]$ then we could,
again in principle, lose if
one job required a long service but only from $w$ while a second job required a
very short service from a large cube of servers centred at $w$. However this
will be rare for small $\lambda $, where our analysis is most relevant.

\section{(Very) greedy lattice animals (GLA)}
\label{VGLA}

As noted, we wish to exploit the celebrated results (see \cite{CGGK}) on greedy
lattice animal systems. Recall that a lattice animal of $\IZ^r $ is simply a
connected subset (when $\IZ^r $ is considered as a graph with the standard
edge set). We are given a collection of i.i.d.\ positive random variables
$\{X(x)\}_{x \in \IZ^r} $. We suppose the existence of $\epsilon > 0$ 
so that $\IE X(0)^{r+1+\epsilon}<\infty$ or, equivalently, the existence
\footnote{The possibility that $X(0)$ have heavy tail justifies
the title of this section, i.e., that animals may be very greedy.}
of $\epsilon >0$ and $C < \infty $ so that for all $t >1$, 
\begin{equation}
\label{suppose}
\IP(X(0) > t ) \le  C/t^{r+1+ \epsilon}.
\end{equation} 
(The $1$ in the power is unnecessary but it is
in this case that we will use our results.) 

We will then parametrize our system
by taking i.i.d.\ random variables $X^\lambda (x) $ to be equal to $X(x)$ with
probability $\lambda $ and otherwise $0$.
For $\zeta \subset \IZ^r $, its $X^\lambda$ value 
(or score) is simply 
\begin{equation}
\label{value}
X^\lambda(\zeta):=\sum_{x \in \zeta} X^\lambda (x).
\end{equation}
The size $|\zeta|$ of the lattice animal is its cardinality.
Note that $X^\lambda(\zeta) \to 0$, as $\lambda \to 0$, in probability,
for any lattice animal of finite size.
We fix positive integer $k$ and $c_1 > 0$ and consider the event 
\[
A_k:=\big\{\exists \text{ lattice animal $\zeta$ containing $0$},\, 
|\zeta|=2^k,\,  X^\lambda(\zeta) \ge c_1 2^k\big\}.
\]
We wish to prove the following upper bound
on the probability of $A_k$, a result which may be of independent interest.
We note that $\IP(A_k)$ depends both on $\lambda$ and $c_1$.
\begin{prop} \label{propGLA} 
Given any $c_1 > 0$, 
there exists a $\lambda_0 >0$ and a
function $C:[0, \lambda _0 ) \to (0,\infty)$ so that 
$C( \lambda ) \rightarrow 0 $ as $\lambda \rightarrow 0$ and so that,
for $\lambda < \lambda_0$ and for all positive integers $k$,
\[
\IP(A_k) \le \frac{C(\lambda)}{2^{k(1+\epsilon)}}.
\]
\end{prop}

\begin{remark}\label{remarkremove}
(i) We can use the above to bound the probability that there is a lattice 
animal of size $u \ge 2^k$ containing the origin whose value is 
$\ge c_1 u$, when $\lambda $ is small,
by considering  $\bigcup_{\ell \ge k} A^{c_1/2}_\ell$ whose probability,
by the above, is less than 
$C(\lambda) 2^{-k(1+\epsilon)} (1-2^{-(1+ \epsilon)})^{-1}$.
\\
(ii)
The above formalism will certainly apply to our situation with random variables
$R_{x,m} + T_{x,m}$ at each site $x \in \IZ^r$.
Indeed, if $X(x)$ denotes the random variable at site $x$ for
rate $\lambda = 1$ conditioned on there being at least one arrival, 
then it is easy to see that with rate $\lambda <1$, 
$R_{x,m} + T_{x,m}$ is stochastically less than $X^{1-e^{-\lambda}}(x)$.
\end{remark}

Some notation used in the proof and elsewhere. 
If $x=(x_1,\ldots,x_r) \in \IR^r$ then
$|x|_\infty := \max_{1 \le i \le r} |x_r|$.
The $L^\infty$ ball $\IB(x,\rho)$ centred at $x$ is the
set 
\[
\IB(x,\rho)= \{y \in \IR^r:\, |y-x|_\infty \le \rho\}.
\]
We also let $|x|_1 := \sum_{i=1}^r |x_i|$.
We use $\I_{A}$ for the indicator of $A$.

\begin{proof}[Proof of Proposition \ref{propGLA}]
We split the value $X^\lambda(\zeta)$, see equation \eqref{value}, 
into three parts:
\begin{equation}
\label{3parts}
X^\lambda(\zeta)  
=  X^\lambda_a(\zeta) + X^\lambda_b(\zeta) + X^\lambda_c(\zeta),
\end{equation}
where
\begin{align*}
X^\lambda_a(\zeta)
&:= \sum_{x\in \zeta} X^\lambda(x)\, \I_{X^\lambda(x) \le 2^{qk}/k^2}
\\
X^\lambda_b(\zeta)
&:= \sum_{x\in \zeta} 
X^\lambda(x)\, \I_{2^{qk}/k^2< X^\lambda(x) \le 2^{vk}}
\\
X^\lambda_c(\zeta)
&:= \sum_{x\in \zeta} X^\lambda(x)\, \I_{X^\lambda(x) > 2^{vk}}.
\end{align*}
The constants $q$ and $v$ appearing in the splitting are
chosen as
\[
q:=\frac{r}{r+ 1 + \epsilon} < v < 1.
\]
Define next four events:
\begin{align*}
A_{k,a} &:=\big\{\exists \text{ lattice animal $\zeta$ containing $0$},\, 
|\zeta|=2^k,\,  X^\lambda_a(\zeta) \ge c_1 2^k/10\big\}
\\
A_{k,b} &:=\big\{\exists \text{ lattice animal $\zeta$ containing $0$},\,
X^\lambda_b(\zeta) \ge c_1 2^k/10\big\}
\\
A_{k,c} &:=\big\{
N^\lambda_c(B_k) \ge m\big\}
\\
A_{k,d} &:= A_k \setminus (A_{k,a} \cup A_{k,b} \cup A_{k,c}),
\end{align*}
where $m$ is a positive integer satisfying
\begin{equation}
\label{mchoice}
m(v(r+1+\epsilon)-r) > 1+\epsilon,
\end{equation}
where $B_k:=[-2^k, 2^k]^r=\IB(0,2^k)$,
and where $N^\lambda_c(B_k)$ is the integer-valued random variable
\[
N^\lambda_c(B_k) := \sum_{x \in B_k} \I_{X^\lambda(x) > 2^{vk}}.
\]
Note that if a lattice animal $\zeta$ of size $|\zeta|=2^k$ contains $0$, then
$\zeta \subset B_k$.

We obtain an upper bound for $\IP(A_k)$ via
\[
\IP(A_k) \le \IP(A_{k,a}) + \IP(A_{k,b}) + \IP(A_{k,c}) + \IP(A_{k,d}).
\]
\noindent
{\bf Bound for $\IP(A_{k,d})$:}
Since $A_k$ occurs, there is a lattice animal $\zeta$ of size $2^k$
containing the origin and having value $X^\lambda(\zeta)\ge c_1 2^k$.
Since $A_{k,a}$ does not occur, we have 
$X^\lambda_a(\zeta) \le c_1 2^k/10$.
Since $A_{k,b}$ does not occur, we have
$X^\lambda_b(\zeta) \le c_1 2^k/10$ and  so
$X^\lambda_b(\zeta) \le c_1 2^k/10$.
Therefore, from \eqref{3parts},
\[
X^\lambda_c(\zeta) \ge \frac{c_1 2^k}{2}.
\]
But this, together with the fact that $A_{k,c}$ does not occur, implies
that there is $x \in B_k$ such that $X^\lambda_x \ge 2^k c_1/2m$.
Hence
\[
\IP(A_{k,d}) \le \sum_{x \in B_k} \IP(X^\lambda(x) \ge 2^k c_1/2m)
\le K_m \lambda/2^{(1+\epsilon)k},
\]
for some constant $K_m$.

\noindent
{\bf Bound for $\IP(A_{k,c})$:} The event $A_{k,c}$ is the event that
the sum of at most $(2\times2^k+1)^r$ Bernoulli random variables, 
each taking value $1$ with probability 
at most $p_k=\lambda C 2^{-kv(r+1+\epsilon)}$, 
exceeds $m$.
To bound this probability we observe that if $S_n(p)$ is
the sum of $n$ i.i.d.\ Bernoulli$(p)$ random variables then
$\IP(S_n(p) \ge m)$ is upper bounded by the probability that there
is a set $A \subset \{1,\ldots,n\}$ of size $m$ such that
all Bernoulli random variables are equal to $1$ on $A$, so
\begin{equation}\label{Bern}
\IP(S_n(p) \ge m) \le \binom{n}{m} p^m \le \frac{n^m p^m}{m!}.
\end{equation}

\[
\IP(A_{k,c}) \le \big( (2^{k+1} +1)^{r} p_k \big)^m  
\le  \frac{K_m' \lambda}{2^{(1+\epsilon)k}},
\]
for some constant $K_m'$ and thanks to the choice \eqref{mchoice} for $m$.

\noindent
{\bf Bound for $\IP(A_{k,b})$:} 
We have
\begin{align*}
\IP(A_{k,b}) &= \IP\left(\exists \text{ lattice animal } \zeta,
\, \sum_{x \in \zeta \cap B_k} X^\lambda(x) \I_{2^{qk}/k^2 
< X^\lambda(x) \le 2^{vk}} \ge \frac{c_1 2^k}{10}\right)
\\
&\le \IP\left( 2^{vk} \sum_{x \in B_k} \I_{2^{qk}/k^2 
< X^\lambda(x) \le 2^{vk}} \ge \frac{c_1 2^k}{10}\right)
\\
&\le \IP\left( \sum_{x \in B_k} \I_{2^{qk}/k^2 < X^\lambda(x)}
\ge \frac{c_1 2^{(1-v)k}}{10}\right)
\end{align*}
The sum in the probability is the sum of $n_k=(2^{k+1}+1)^r$
Bernoulli random variables, each with probability being 1 being
at most $p_k:= \lambda C (2^{qk}/k^2)^{-(r+1+\epsilon)}=
\lambda C  k^{2(r+1+\epsilon)}/2^{rk}$.
We can apply now inequality \eqref{Bern} with $n=n_k$, $p=p_k$ and
$m=c_12^{k(1-v)}/10$, and the Stirling formula for $m!\sim \sqrt{2\pi m}e^{m (\log m -1)}$, 
to obtain that the required probability $\IP(A_{k,c})$ is not bigger than
$$
\lambda e^{-m (\log m - 2(r+1+\varepsilon ) \log k) (1+o(1))} 
= \lambda o \left(2^{-(1+\varepsilon )k}\right),
$$
as $k\to \infty$. Therefore, 
\[
\IP(A_{k,c}) \le  \frac{K'' \lambda}{2^{(1+\epsilon)k}},
\]
for some constant $K''$.

\noindent
{\bf Bound for $\IP(A_{k,a})$:}  
We repeat the argument given in \cite{JM} (or \cite{CGGK}).  
\begin{lemma}[Lemma 1 in \cite{CGGK}, Lemma 2.1 in \cite{JM}]
\label{combinlemma}
For any lattice animal $\zeta $ of size $n$ containing the origin and
any $1 \le \ell \le n$ we can find a sequence $0 = u_0, u_1, \ldots, u_h$ 
of points in $\IZ^r$, $h$ being the integer part of $2n/\ell$,
and $|u_i-u_{i-1}|_{\infty}\le1$ for all $1 \le i \le h$, so that 
\[
\zeta  \subset  \bigcup _{i=0}^ h \IB(\ell u_i, 2\ell).
\] 
\end{lemma}
\begin{proof}

For $y=(y^1,\ldots,y^r)\in \IR^r$, let $\lfloor y/\ell \rfloor$ 
be the point $x$ in $\IZ^r$
such that $x^i = \lfloor y^i/\ell \rfloor$ (the quotient of the division
by $\ell$, componentwise). Clearly, $\ell x \le y < \ell (x+1)$, componentwise,
so $|y-\ell x|_\infty \le \ell$.
If $\zeta$ is a lattice animal containing $0$ we can find a sequence
$\pi=(\pi_0,\ldots,\pi_{2n})$ such that successive elements are either
identical or neighbors in $\IZ^r$ ($\pi$ is a path) and such that 
$\{\pi_0,\ldots,\pi_{2n}\}=\zeta$.
(To do this, consider a spanning tree of $\zeta$ and form $\pi$ by
traversing the tree ``from the bottom''.)
Then $|\pi_i-\pi_j|_\infty \le \ell$ if $|i-j| \le \ell$.
Define $u_i \in \IZ^r$ by
$u_i := \lfloor \pi_{i\ell}/\ell \rfloor$, $i=0,\ldots, h$.
Then $|u_i-u_{i-1}|_\infty \le \ell$ for all $i$. Furthermore, 
if $x \in \zeta$ then $x=\pi_t$ for some $0 \le t \le 2n$. 
Let $k=\lfloor t/\ell \rfloor$. Then 
$|\pi_t-\ell u_{k}|_\infty \le |\pi_t-\pi_{k\ell}|_\infty
+ |\pi_{k\ell}-\ell u_{k}|_\infty \le \ell+\ell$,
so $x=\pi_t \in \IB(\ell u_k,2\ell)$.
\end{proof}
From this it is immediate that for given $\ell$ there are at most 
$9^{r 2n/\ell}$ 
such $2\ell$--ball coverings.  
We use this result with $n= 2^k $ .
We consider $\ell$ of ``scale'' $2^i$ with $2^i \le  2^{qk}/k^2$.  
Let $i_0 $ be the maximal such value.  
For given $i$, we choose the value
$\ell = l(i)$ to equal the integer part of $\lambda^{-1/2r}\, 2^{i/q}$.  
With this value the probability that a $L^\infty$ ball 
of radius $2\ell$ contains a site having an $X^\lambda $ value is small 
for $\lambda $ small but not (in principle) negligible.  
From this it is easily seen that given a sequence $u_0, u_1,\ldots,u_h$ 
satisfying the above (and therefore given an $\ell$ covering), 
the probability that
\[
\text{the number of sites within the covering having value at least $2^i $ 
is at least $2(2^k/\ell) c_1 $}
\]
is bounded above by
$20^{-2 r \cdot 2^k/\ell}$ for $\lambda $ small.  
Thus we see that outside  an event of  probability $9^{hr} 20^{-2 \cdot 2^kr/\ell}$,  
this bound will hold for all $\ell$--coverings.  
Summing over $i$ such that $2^i  \le  2^{kq}/k^2$
we have that outside probability 
\[
\sum_ {2^i \le  2^{kq}/k^2}  (\frac12) {2 \cdot 2^kr/\ell(i)}  
\le 2 (\frac12)^{2 \cdot 2^kr/\ell(i_0)}  
\le 2 (\frac12)^{c(\lambda k^{2/q})} 
\]
for each such $i$ and for each corresponding $\ell(i)$--covering, 
the number of {sites  in the covering whose $X^\lambda$  value 
at least $2^i$} is at most $2(2^k/\ell)c_1$.

Thus (outside of probability $2 (\frac12)^{c(\lambda k^{2/q})}$ 
for some universal $c$) we have, for any lattice animal $\zeta$
of size $2^k$,  
\[
\sum_{x \in \zeta} X^\lambda(x) 
\I_{X^\lambda(x) \le \frac{2^{k(r+1+ \epsilon)/d}}{k^2} } 
\le \sum_{i \le i_0} 2(2^k/\ell(i))\, c_1 2^{i+1}   
\le \sum_{i \le i_0} 4\cdot 2^k \lambda^{1/2r} 2^{-i(1+\epsilon)/r} 
\]
which is bounded by 
$\text{Constant}(\epsilon) 2^k \lambda^{(1+\epsilon)/2r}$.  
The conclusion follows for large $k$.
Thus we have shown the proposition.
\end{proof}

\begin{cor} \label{corcube} 
Define
\[
B^{c_1}_u(x):=
\{\exists \text{ lattice animal $\zeta$ containing $x$,\, 
$|\zeta| \ge u$,\, $X^\lambda(\zeta) \ge c_1 |\zeta|$}\}.
\]
For $c_{1} < 1$ fixed, there exists a constant $\lambda_1 = \lambda_1 (c_1) $
and a function $H$ defined on $[0, \lambda_1 )$ tending to zero as $\lambda $
tends to zero, so that for all $0 < \lambda < \lambda_1$ and all 
positive integers $R$,
\[
\IP\left(\bigcup_{x \in [-R,R]^r} B^{c_1}_u(x)\right)  \le
\begin{cases}
\displaystyle \frac{H(\lambda )}{{(u+1)}^{1+ \epsilon}}, & u \ge R
\\
\displaystyle \frac{H(\lambda ) R^{r}}{{(u+1)}^{r +1+ \epsilon}}, & u \le R.
\end{cases}
\]
\end{cor}

\begin{proof}
We treat the case $u\ge R $ only as that for $u \le R$ is essentially the same.
Fix $c_2 < c_1$. Let 
\[
N = \sum_{x \in [-R, R]^r} \I_{B^{c_2}_u(x)}.
\]
By Proposition \ref{propGLA} and Remark \ref{remarkremove}(i),
$\IP(B^{c_2}_u(x)) \le C(\lambda)/2^{k(1+\epsilon)}$ where
 $k$ is the largest integer with $2^k \le u$.
Therefore
\[
\IE(N) \le \frac{C(\lambda)(2R+1)^r}{2^{k(1+\epsilon)}}.
\]

Now suppose that event  $\bigcup_{x \in [-R,R]^r} B^{c_1}_u(x) $  occurs.  
Then for some $x \in [-R, R]^r$ and   some lattice animal $\xi$ containing $x$, 
$X^\lambda (\xi) \geq c_1 | \xi |$.  
Now for every $y \in  [-R, R]^r$ with $|y-x|_\infty\le \frac{R(c_1-c_2)}{c_2r}$, 
we can create a new lattice animal $\xi'$ 
containing both $y$ and $\xi$ by adding at most 
$(c_1-c_2)R/c_2$ points to $\xi$. Since we assumed $R \le u \le |\xi|$, we have
$|\xi| \leq |\xi'| \leq (c_1/c_2) |\xi|$.
By positivity of the random variables 
\[
X^\lambda (\xi') \geq  X^\lambda(\xi) \geq c_1 | \xi |  \geq  c_2 | \xi'|.
\]
Thus the event $\bigcup_{x \in [-R,R]^r} B^{c_1}_u(x) $ is a subset of the event that random variable $N$ defined at the start of the proof is at least
$(\frac{(c_1 -c_2)R}{rc_2})^r$.  Our result now follows from Markov's inequality.

\end{proof}

\section{Cluster formation and their properties}
In this section we construct clusters for our Poisson hail corresponding to
integer intervals $(m-1,m]$. The clusters themselves will follow a clustering
procedure of \cite{BF} and will depend only on the 
random variables $\{R_{x,m}\}_{x}$. 
Our departure will consist in the temporal (or workload) variable 
we associate to each cluster.
Our clusters will have the property that if $C \subset \IZ^d $ is a cluster and
$\gamma : (m-1,m] \to \IZ^d$ is a path satisfying property (ii) of 
Definition \ref{defgamma}, then
\begin{equation}
\label{clusterprop}
\gamma(m)  \in  C \Rightarrow \gamma(s) \in C \text{ for all } s \in (m-1,m].
\end{equation}

Recall we discretized time by identifying with $m$ all tasks for site $x$
arriving in $(m-1, m]$ with a single task of ``radius''
\[
R_{x,m} \equiv \sum R_i^{x,m},
\] 
summed over all tasks arriving at $x$ in time
interval $[m-1,m]$.
We denote by $t^{x,m}_i $ the times of the arrivals, i.e., the points of the 
Poisson process $N_x$
in the interval $(m-1, m]$. The indices $i$ are
coordinated so that for site $x$ a job arrives at time $t^{x,m}_i$ requiring
$\tau^{x,m}_i$ units of service from
servers in $x + [-R_{i}^{x,m} , R_{i}^{x,m}]^d$.

By Lemma \ref{basic}, 
\[
\IP(R_{x,m} \ge u) \le \frac{C}{(u+1)^{d+1+\epsilon}},
\]
for some $\epsilon >0$ and some finite constant $C= C( \epsilon)$ 
for any $ \epsilon $ conforming to the hypotheses of Theorem \ref{thm1}.

For fixed ``time'' $m$ and $y\in \IZ^d$ let 
\[
D_{y,m}:=\IB(y, R_{y,m})
\]
be the $L^\infty$ ball centred at $y$ and having radius $R_{y,m}$.
The cluster $C(x,m)$ containing $x$ is defined as the union of such $D_{y,m}$ 
over $y$ having the property that there exists integer $K$
and sites $y=z_0, \ldots, z_K=x$
such that $D_{z_i,m} \cap D_{z_{i-1},m} \not = \varnothing$,
for all $1\le i \le K$. 

Let $D(x,m) $ be the diameter of the cluster $C(x,m)$.
It is clear that these clusters have the property (\ref{clusterprop}) above.
What is not a priori clear is that even with very small rate $\lambda$ 
the clusters will be a.s.\ finite. 
However the preceding section enables us to prove
\begin{lemma} \label{lemDIAM}
Assume that $\IP(X(0)>t) \le C/t^{d+1+\epsilon}$.

Then there exists a function $K(\lambda)$ tending to zero as 
$\lambda$ tends to zero so that, for $\lambda$ sufficiently small and 
all positive integers $z$,
\[
\IP(D(0,m) \ge z)  \le \frac{K(\lambda)}{z^{1+\epsilon}}.
\]
\end{lemma}
\begin{proof}
Consider the GLA system with random variables $\{Z(x)\}_{x \in \IZ^d}$ for 
\[
Z(x)  =  R_{x,m} .
\]
If the diameter $D(0,m)$ of the cluster $C(0,m)$ containing the origin
exceeds $z$ then there must exist $L$ and
$0 = x_0, x_1, \ldots, x_L$ so that, for all $1 \le i \le L$,
\begin{equation}
\label{noted}
|x_{i-1} - x_i|_\infty \le R_{x_i,m} + R_{x_{i-1},m}
\end{equation}
and $|x_L| \ge z$.
We choose $\zeta$ to be the lattice animal 
$\bigcup_{i=1} ^ L P(x_{i-1}, x_i) $ where $P(x_{i-1}, x_i)$ is 
a path connecting $x_{i-1} $ and $x_i $ of length $|x_{i-1} - x_i |_1$.
(Recall that $|x|_1$ denotes the $L^1$ norm.)
Then $\zeta$ is a lattice animal in $\IZ^d$ containing the origin 
for which $\sum_{y \in \zeta} Z(y)  \ge \sum_{i=0}^L Z(x_i)$.
By \eqref{noted}, 
\[
\sum_{i=0}^L Z(x_i) 
\ge \frac12 \sum_{j=1}^L |x_{i-1} - x_i|_\infty 
\ge \frac{1}{2d} \sum_{j=1}^L |x_{i-1} - x_i|_1 \ = |\zeta  |/2d \ge z/2d.
\]

The result follows from
Proposition \ref{propGLA} applied to $c_1 < \frac{1}{4d}$.

\end{proof}

Arguing as in Corollary \ref{corcube}, we obtain 

\begin{cor} \label{corclustercube}
There is a function $C(\lambda)$ tending to zero as $\lambda \to 0$ so that
for $\lambda $ small, for all $L$, and for $R \le L/2$,
\[
\IP(\exists x \in [-R,R]^d \text{ with } D(x,m) \ge L)  \le  \frac{C(\lambda)}{L^{1+ \epsilon}}.
\]
while for $ \lambda $ small and $R \ge L/2$
\[
\IP(\exists x \in [-R,R]^d \text{ with } D(x,m) \ge L)  \le  \frac{C(\lambda) R ^ d }{L^{d + 1+ \epsilon}}.
\]
\end{cor}

We now consider the ``time'' $T(x,m)$ associated with
the cluster $C(x,m)$.
This definition is a little less direct than that for $D(x,n)$: 
Given $x \in \IZ^d$ and integer $m$ (and so given cluster $C(x,m)$), 
$T(x,m)$ is equal to the maximum value of
\[
\sum_{i=0}^L \tau_{j(i)}^{x_i,m}
\]
over sequences $x_0, x_1, \ldots, x_L  \in C(x,m)$ and 
$m \ge t_0 \ge t_1 \ge \cdots \ge t_L \ge {m-1}$ so that,
for all $i$ a job arrives at $x_i$ at time $t_i = t^{x,m}_{j(i)}$ having
work time $\tau^{x_i,m}_{j(i)}$ and
\[
|x_{i-1} -x_i|_\infty 
\le R^{x_i,m}_{j(i)} + R^{x_{i-1},m}_{j(i)}, \quad 1\le i \le L.
\]

We remark that, under the latter two conditions, if $x_0 \in C(x,m) $ then
necessarily the ``subsequent" $x_i$ are also in this cluster.
We note also that this definition (which requires more information than the
discretized data) ensures that, for any site in $C(x,m)$, 
the waiting time accrued during time interval $(m-1,m]$ is 
less than or equal to $T(x,m)$.

\begin{lemma} \label{lemTIME}
There exists function $K(\lambda)$ which tends to zero as $\lambda$ tends 
to zero so that for $\lambda$  sufficiently small for all $z \ge 1$,
\[
\IP( T(0,m) \ge z )  \le  \frac{K( \lambda )}{z^{1+ \epsilon }}.
\]
\end{lemma}

\begin{proof}

By the previous lemma we may suppose that $D(0,n) \le z/100$.
Again if $T(0,m)$ takes a value exceeding $z$ then there must exist $L$ and 
a  sequence
$x_0,x_1, \ldots, x_L \in C(0,m)$ and times
$m \ge t_0 \ge t_1 \ge \cdots \ge t_L \ge {m-1}$ 
so that, for all $i \le L-1$,
there is a job arrival at $x_i $ at time $t_i$ and
\[
|x_{i-1} - x_i |_{\infty} \le R^{x_i,m}_{j(i)} + R^{x_{i-1},m}_{j(i-1)},
\quad 1 \le i \le L,
\]
and also $\sum_i T(x_i,t_i ) \ge z$.
It is important to note that we do not assume that the $x_i$ are distinct.
Indeed it is for this reason that we use that bound involving 
$R^{x_i,m}_{j(i)}$ rather than $R_{x_i,m}$. 
However if $y $ is equal to 
$x_{i_1}, x_{i_2}, \ldots, x_{i_r}$, then of course 
$R_{y,m} \ge \sum_k R^{x_{i_k},m}_{j(i_k)}$ and
equally $T_{y,m} \ge \sum_k \tau^{x_{i_k},m}_{j(i_k)}$. Thus as before we
obtain, as in Lemma \ref{lemDIAM}, but with 
$Z(x) = R_{x,m} + T_{x,m}$,
that for a lattice animal 
$\zeta = \bigcup_i P(x_{i-1}, x_i)$ that the
GLA$(Z)$ score (i.e.\ $\sum_{x \in \zeta} (R_{x,m} + T_{x,m} ) $) will exceed $(z+|\zeta|)/4d$.

\end{proof}

Again we have 

\begin{cor} \label{cortimecube}
There exists function $C(\lambda)$ which tends to zero as $\lambda $ 
tends to zero so that
for  so that for all $L$ and for $R \le L/2$,
\[
\IP(\exists x \in [-R,R]^d \text{ with } T(x,m) \ge L)  \le  
\frac{C( \lambda)}{L^{1+ \epsilon }}.
\]
while for $ \lambda $ small and $R \ge L/2$
\[
\IP(\exists x \in [-R,R]^d \text{ with } T(x,m) \ge L)  \le  \frac{C(\lambda) R ^ d }{L^{d + 1+ \epsilon}}.
\]
\end{cor}

\section{Workload bounds and stability}
We now apply the foregoing to analyze the workload stability for small values of
$\lambda$. It is enough to show tightness of the workload $W^n(0,0)$
at time $0$ when the system starts empty at time $-n$. Recall that $W^n(0,0)$ is
obtained as the maximum of scores $V(\gamma)$ where $\gamma$ ranges in the
set of paths $\Gamma^n(0,0)$. See Definition \ref{defgamma}.

Due to the monotonicity properties of the system, 
$W^n(0,0)$ is readily seen to be bounded above by the
quantity $W^{n,D}(0,0)$ which corresponds to the discretized system 
and is given by
\[
W^{n,D}(0,0)  =  \sup_\gamma V^D(\gamma), 
\]
where the supremum is taken over discrete time indexed paths paths 
$\gamma : [-r,0] \rightarrow \IZ^d$ for some $0 \le r \le n$
satisfying\\
(i) $\gamma (0)  =  0$, 
\\
(ii) for each $-n < i \le 0$, $\gamma(i-1)$ 
belongs to cluster $C(\gamma(i), i)$.
\\
The score $V^D(\gamma)$ of $\gamma$ is given by
\[
V^D( \gamma )  =  \left( \sum_{i=0}^{r-1} T(\gamma(-i),-i)) \right) - r.
\]
We now consider a cube $H$ of length $R$ in $\IZ^{d+1}  =  \IZ^{d} \times
\IZ $ where the first $d$ coordinates are considered as ``spatial" and 
the last one temporal.
Accordingly, we write $H $ as $H^\prime \times I$ where $I$ is a
temporal interval of length $R$ and $H^\prime$ is a cube of 
length $R$ in $\IZ^d$. We define the variable 
\begin{multline*}
V(H,u):= \text{number of clusters $C(x,m)$ intersecting $H' \times \{m\}$}
\\
\text{and having $D(x,m)+T(x,m) \ge u$, $m \in I$.}
\end{multline*}
Clusters are by definition enclosed
in a slab $\IZ^d \times \{m\} $ for some $m$ and the clusters at different
temporal levels $m$ are independent. 
Thus (after repeated use of Lemma 2 of \cite{BF})
we easily obtain from Corollaries \ref{corclustercube} and \ref{cortimecube}.

\begin{prop} \label{James}
There exists constant $K_\lambda $ so that, for $u \le R$, $V(H,u) $ is
stochastically less than Poisson of parameter 
$\frac{K_\lambda R^{d+1}}{{(u+1)}^{d +1+ \epsilon}}$.
For $u \ge R$ it is bounded by a Poisson of parameter 
$ \frac{K_\lambda R}{{(u+1)}^{1+ \epsilon}} $.
Furthermore, as $\lambda $ tends to zero, $K_\lambda $ tends to zero.
\end{prop}

Given the above for $C$ a cluster corresponding to temporal interval $(m-1,m]$, 
we define a value $X_m(C)$ to signify the value
$D(x,m) + T(x,m)$ for a (and so any) $x $ in $C$.

In analyzing $V^D( \gamma ) $ we will consider a (nonstandard) 
lattice animal system on  $\IZ^{d+1}$.  
Instead of having i.i.d.\ random variables 
indexed by points $(x,m) \in \IZ^{d+1}$, we will consider 
a lattice animal model based on the random variables $X_m(C)$ which, 
while independent for  distinct collections of index $m$, are not independent.  
Given a lattice animal $\Xi \subset \IZ^{d+1}$, we write $\Xi_m$  
to denote $\Xi \cap \IZ^{d} \times \{m\}$ ( of course in general $\Xi_m$ will not be a lattice animal).  
Obviously given $\Xi$, all but finitely many $\Xi_m$ will be empty.  
The value $V^C(\Xi)$ associated with such a lattice animal $\Xi$ will be
\[
V^C(\Xi)   := \sum_{m\in \IZ}  \,
\sum_{C \text{ cluster in } \IZ^d \times\{m\}} X_m(C) \,
\I_{C\cap \Xi_m \ne \varnothing}.
\]
To analyze $\sup V^C(\Xi)$ over all lattice animals containing the origin 
of $\IZ^{d+1}$ and of cardinality $N$ we proceed as in Section \ref{VGLA}.
For a positive integer $k$ we let
\[
L_k := 2^{k(1+\epsilon /2(d+1))}.
\]
(We are primarily interested in $k$ with $L_k \le N/\log^2(N)$).
We know from Section \ref{VGLA} that there are less than
$K_d^{2N/L_k+1}$ collections of $L_k$ cubes in $\IZ^{d+1}$, 
each collection denoted as $\{C^k_1, C^k_2, \ldots, C^k_{2N/L_k + 1}\}$, 
so that each $\Xi$ considered is contained in the union of the $C^k_i$ 
for one of these collections.    
Given such a collection we have, by Proposition \ref{James}, that for any 
$j$, $V(C^k_j, 2^k)$ is stochastically less than a Poisson random variable 
of parameter $K_\lambda/2^{\epsilon / 2}$.  
Furthermore (again by Lemma 2 of \cite{BF}), we have that,
having identified all clusters intersecting $\bigcup_{i<j} C^k_i$, 
then conditional number of ``extra'' clusters intersecting $C^k_j$ 
is stochastically less than this Poisson random variable.   
Thus, just as in Section \ref{VGLA}, we obtain the following:
if $N_k(\Xi)$ is the number of clusters of value more than 
$2^k$ that intersect $\Xi$, then, for all $N/L_k \le \log^2(N)$ and $N$ large,
\[
\IP(\sup _\Xi N_k(\Xi) \ge 3K_\lambda N/L_k ) \le e^{-N/L_k}. 
\]
Thus for every lattice animal $\Xi$ of size $N$ containing the origin, 
outside of probability bounded by $\text{Const} \times e^{-\log^2(N)}$, 
the contribution to $V^C(A)$ from clusters $C$ having 
$X_n(C) $ less than 
$\left( \frac{N}{\log^2(N)} \right)^{\frac{1}{1+ \epsilon /2(d+1)}} = N_0$ 
is less than
\[
3 N K_\lambda  \sum _{2^k \le N_0} 2^k/2^{k(1+ \epsilon /2(d+1) )} 
\le N K_\lambda C( \epsilon),
\]
for some finite $C(\epsilon)$, where (we stress) $K _ \lambda $ tends to zero as $\lambda $ tends to zero.

Given this bound we easily deal with the clusters having value greater 
than $N_0$ using Proposition \ref{James} and the arguments of 
Section \ref{VGLA} and obtain
\begin{prop} 
\label{newb}
For the above system, for each $\delta>0$, there exists a sufficiently
small $\lambda >0$, such that the probability that 
there exists a lattice animal $\Xi$
of size at least $n$, containing the origin of $\IZ^{d+1}$ and with 
$V^C(\Xi) > \delta |\Xi|$ is less than 
$C_\delta / n^{\epsilon/2} $ for all positive integers $n$.
\end{prop}

We now apply this to the values $V^D(\gamma)$.
We denote by $\Gamma_m$ the set of discrete time paths 
$\gamma : [-m,0] \rightarrow \IZ^d$ 
satisfying the stipulated conditions: $\gamma(0) = 0$
and for all $0 \le i < m$, $\gamma(i-1) \in C(\gamma(i),i)$.  
It is immediate that if a curve (in continuous time, $\gamma[-m,0]
\rightarrow \IZ^d$, is in $\Gamma^m(0,0)$, then its ``skeleton'' 
$\gamma(-m), \gamma(-m+1) , \ldots, \gamma(0) = 0$ is in $\Gamma_m$.
\begin{prop} 
\label{propdef}
There exists $\lambda_0 >0$ and $C < \infty$ such that for all $n \geq 1$ and
for all $\lambda < \lambda_0$,
the probability that there exists an $m \ge n$ so that
$V^D(\gamma) > -m/2$, for some $\gamma \in \Gamma_m (0,0)$,  
is bounded by $C/n^{\epsilon/2}$.

\end{prop}
\begin{proof}

We associate to each path $\gamma(-m), \gamma(-m+1), \ldots, \gamma(0)$ 
the score 
\[
\sum_{j=-m+1}^0 T(\gamma (j),j) + \sum_{j=-m+1}^0 |\gamma(j)-\gamma(j+1)|_1.
\]
Note that, by definition of $T(x,n)$, for each $j$, the sum of  durations 
for jobs  which arrive at time $s \in(j-1,j]$ and so that the job 
requires service from both server $\gamma(s)$ and $\gamma (s-)$ 
must be less than $T(\gamma(j),j) = T(\gamma(j-1),j)$.  
Thus in particular for a continuous time curve $\gamma \in \Gamma^m(0,0)$,
\[
V^D(\gamma) + m  \le  \sum_{j=-m+1}^0 T(\gamma (j), j)   
\le \sum_{j=-m+1}^0 T(\gamma (j),j) 
+ \sum^{j=0}_{-m+1} |\gamma(j)- \gamma(j+1)|_1.
\]
We can associate the path $\gamma $ with the ``lattice animal'' $\Xi^\gamma $
in $\IZ^{d+1}$ consisting of points
$(\gamma(i),i)$ for $i = -m,-m+1, \ldots 0$ 
together with for each $-m < i \le 0$
the points $(y,i)$ which lie on a path $P_i$ from
$(\gamma(i),i) $ to $(\gamma(i+1),i)$ which lies within $\IZ^d \times \{i\}$ 
and has length $|\gamma(i) - \gamma(i+1)|_1$. 
Thus this lattice animal $\zeta$ has size
\begin{multline*}
|\zeta|:= m+1  +  \sum_{i=-m+1}^0 (|\gamma(i) - \gamma(i-1)|_1-1)_+  
\\ \le  m +1
+ d \sum_{i=-m+1} ^ 0 D(\gamma(i) , i )  =  m +1
+ d \sum_{i=-m+1} ^ 0 D(\gamma(i-1) , i ).
\end{multline*}
We note that, while obviously $|\Xi^\gamma | \ge m+1$, there are no 
nonrandom upper bounds for the cardinality.
Thus the inequality above can be rewritten as 
\[
V^C(\Xi^\gamma)  \ge V^D(\gamma ) + m , \quad 
V^C(\Xi^ \gamma ) \ge \frac{1}{d} \sum_{j=-m+1} ^
{0} T(\gamma (j), j)  +  \sum^{j=0}_{-m+1} |\gamma(j)- \gamma(j+1) |_1.
\]
We now invoke Proposition \ref{newb} with $\delta = \frac{1}{20d} $ 
to deduce that (with $\lambda $ sufficiently small) the (bad) event 
\[
B = \{\exists \text{ lattice animal }  \Xi,\, 0 \in \Xi,\, |\Xi| \ge m,\, 
V^C(\Xi) \ge \delta |\Xi|\}
\]
has probability bounded by $C/m^{\epsilon/2} $ for some universal $C$.
Our analysis of now splits into two cases. In both cases we suppose that 
bad event $B$ does not occur.
\\
Firstly suppose that $|\Xi^\gamma| \le 4dm+m$.  In this case
\[
V^D(\gamma) \le \sum_{j=-m+1}^0 T(\gamma (j), j)  -  m.
\]
But, on event $B^c$, $\sum_{j=-m+1}^0 T(\gamma (j), j) 
\le V^C(\Xi^\gamma) \le \frac{4dm+m}{20d}$. This implies that 
\[
V^D(\gamma) \le -3m/4.
\]
On the other hand, suppose that $|\Xi^\gamma| > 4dm+m$.  Now we have 
\[
V^C(\Xi^\gamma ) )  \ge \sum^{j=0}_{-m+1} 
|\gamma(j)- \gamma(j+1)|_1 / d \ge |\Xi^\gamma|/2d.
\]
But this is impossible on event $B^c$.  
\\
Thus we have shown that, on event $B^c$, for all $m \ge n$, the event 
$A_m = \{ \exists \gamma \in \Gamma^m(0,0)\, V^D( \gamma ) > -2m/3\}$ 
does not occur, and so
$\IP(A_m) \le C/m^{\epsilon/2}$.
So 
$\IP(\bigcup_{\frac65^m\ge n} A_{\frac65^m})\le C^\prime/n^{\epsilon/2}$.
But it is easily seen that the event 
$\{\bigcup_{\frac65^m \ge n} A_{\frac65^m}\}$ contains the event 
$\bigcup_{m\ge n} \{\exists \gamma \in \Gamma^m(0,0)\, V^D(\gamma) \ge -m/2\}$.
\end{proof}

\begin{proof}[Proof of Theorem \ref{thm1}]
From Proposition \ref{propdef} we have that the workload $W^{n,D}(0,0)$
of the discretized system is tight as $n $ varies.
On the other hand, by monotonicity, the limit as 
$n \to \infty$ of $W^{n}(0,0)$ exists a.s.
Tightness ensures that this limit is finite. 
If we now start the system in full vacancy at time $0$
and consider the workload $W(x,n)$ for some $n >0$, 
we have that $W(x,n)$ is in distribution
equal to $W^n(0,0)$. 
Therefore $W(x,n)$ converges in distribution as $n \to \infty$.
\end{proof}

\begin{remark}
We have actually shown something stronger than tightness: namely, that, 
starting with an initially empty system, the workload profile at time $t$
converges in distribution, as $t \to \infty$, to some distribution 
which we will denote
by $\mu$. Standard arguments show that $\mu$ is an invariant measure:
if we start with $W(0, \cdot)$ distributed according to $\mu$ then
$W(t,\cdot)$ also has distribution $\mu$.
Since $W(t,\cdot)$ is translation invariant in space, this is the case 
for the limit $\mu$. 
We have thus proved the existence of an invariant probability measure which
is also spatially invariant. 
\end{remark}

\section{Necessity and proof of Theorem \ref{thm2}}
Let  $0 < \epsilon < d+1$. We consider the case where  the stone heights 
(job service times) $\tau$ satisfy, for $t$ positive integer,
\[
\IP(\tau \ge t)  =  \frac{1}{t^{d+1-\epsilon}},
\]
and the  stone basis $B$  is the cube
\[
B = [-\tau, \tau]^d .
\]
We consider the number of job arrivals in space time cube $[0, t)^{d+1}$ of
duration at least $2t$ for integer $t$: that is
the number of arrivals $(B, \tau)$ so that \\
(1) $\tau \ge 2t$,\\
(2) $B$ is a cube of side length $2 \tau + 1$ centred at a site in $[0,t)^d$
and\\
(3) the job arrives at a time in $[0,t)$.

This random variable has expectation $t^\epsilon \lambda / 2^{d+1- \epsilon} $
which for fixed $\lambda $ tends to infinity as $t$ becomes large.
In particular for $t$ large this expectation strictly exceeds $\frac12 $. We
fix such a $t$ now.

Obviously this applies to any translation of the cube and the random variables
associated to disjoint space time cubes are independent.

We consider the path $\gamma^n $ in $\Gamma^{nt}(0,0) $ which is identically
the origin for all $s \in [-tn, 0]$.
We note that under our assumptions on $(B, \tau)$ any job arriving at a site in
$[0,t)^d$ having $\tau \ge 2t$ requires service from the origin. Hence the path
$\gamma^n$ has value at least
\[
\sum_{j=1}^n 2t X_j - nt,
\]
where $X_j$ is the number of jobs arriving during interval $(-jt, -(j-1)t]$
and satisfying (1) and (2) above. By the law of large numbers, $V(\gamma^n )$
tends to infinity a.s., as $n$ tends to infinity.
This is enough to establish instability of the workload in this case no matter
what the value of $\lambda >0$ might be.

\begin{remark}
In fact this argument can easily be generalised to show that if 
for each arrving job, $\tau = R$ and if, for some $\epsilon > 0$,
$E( \tau ^{d+1- \epsilon } ) = \infty $, then the system does not have stability.
\end{remark}

\section{Uniqueness}
We now briefly address the question of unicity of invariant measures for the
workloads when the power law condition holds and when $\lambda$  is
sufficiently small.
We know that if condition \eqref{maincondition} is satisfied and 
parameter $\lambda $is sufficiently small 
then the distribution $\mu$ of workloads, obtained by
starting the system at time $-n$ with the workloads identically zero and letting
$n \rightarrow \infty$, is invariant. The question that naturally arises is
whether other equilibria for the workload, under Poisson arrival of jobs, are 
possible.

We consider systems that are stationary under spatial translations and show the
following.

\begin{decth} \label{unique}
Under the condition \eqref{maincondition} above, there exists $\lambda_0 $ 
so that if the arrival
rate $\lambda $ is less than $\lambda_0$, and $\nu $ is an invariant probability
for the system on the space of workloads that is preserved by spatial
translation, then $\nu  =  \mu$.
\end{decth}

In this section, the assumption that all jobs require service from cubes 
of servers is not ``without loss of generality'' 
so we remark that we only use the weak ``irreducibility''
condition that for every  neighbour $e$ of the origin
there exist  sequences
\[
0  = x_0, x_1, \ldots, x_r  = e
\]
and bases
\[
B_1, B_2, \ldots, B_r
\]
so that for all $i$,  $x_{i-1}, x_i \in B_i$ and jobs $B_i$ 
occur with strictly positive probability.

To show the claimed uniqueness
it suffices to show that for such a measure $\nu$ and any bounded
cylinder function $h$, we have
\[
\int h(\eta) \nu ( d \eta)  =  \int h(\eta) \mu ( d \eta).
\]
Assuming that $\nu$ is invariant
this is equivalent to 
\[
\IE^{\nu} [h(W_n)]  =  \int h(\eta) \mu ( d \eta)
\]
for any $n$ (and so, in particular, for $n$ large). 
Given this and our construction of the measure $\mu$,
it will be enough to show that for $\epsilon' > 0$ and $h$ as above, both fixed,
\[
\left|\IE^\nu[h(W_n)] - \IE^{\vec 0}[h(W_n)] \right| < \epsilon',
\]
for $n$ large.  This will be our objective in the following.

As $\nu$ is temporally invariant, we have, by the ergodic theorem
\cite{K} that, 
for every $M$, a.s.,
\[
\lim_{t \rightarrow \infty} \frac{1}{t} \int_0^t \I_{W^0_s \le M}\, ds  
=  \nu(W: W^0 \le M \vert \mathcal{I}_T ),
\]
where $\mathcal{I}_T$ is the $\sigma $-field of events 
that are invariant under temporal shifts.
Thus, for an $\epsilon > 0$ fixed, we can find an $M$ so large that,
\[
\frac{1}{t} \int_0^t \I_{W^0_s \le M}\, ds  \ge 1- \epsilon, \quad 
\text{ for all } t > M,
\]
with probability at least $1-\epsilon^2$.  
By (spatial) translation invariance we then have for (every) $x \in \IZ^d$
\[
\frac{1}{t} \int_0^t \I_{W^x_s \le M}\, ds  \ge 1- \epsilon, \quad 
\text{ for all } t > M,
\]
with probability at least $1-\epsilon^2$.
Let us call the above event $B^x_M$. Thus we will have, by the ergodic theorem 
applied to spatial shifts, that
\[
\lim_{k \rightarrow \infty}  \frac{1}{(2k+1)^d}\, \sum_{|x| \le k} \I_{B^x_M}  
=  \nu(B^0_M  \vert \mathcal{I}_S),
\] 
where $\mathcal{I}_S$ is the sigma field of spatially shift invariant events.
Thus we have that, for $k_0$ sufficiently large,
with probability at least $1- 2\epsilon$,
\[
\frac{1}{(2k+1)^d} \sum_{|x| \le k} \I_{B^x_M} \ge 1-\epsilon, \quad k \ge k_0.
\] 
We now note that at time $0$, say, the existence of a large workload $V$ at a
site $0,$ say, implies that with reasonable probability the workload will be of
order $V$ for a time of order $V$
in the time interval $[0,V]$ for a cube of sites of side length of order $V$.
\begin{prop} 
\label{propog}
There exists $c_1 \in (0, \infty)$ so that, for all $V$ large enough, uniformly
over initial workloads $W(0,\cdot)$ with $W(0,0)>V$,
with probability at least $c_1$, we have, for all $x \in [-c_1V,c_1V]^d$,
\[
\quad W(x,t) > V/4, \text{ for all } t \in [V/2, 3V/4].
\]
\end{prop}
\begin{proof}
From our ``irreducibility" assumptions on the distribution of jobs, it is clear
that there exist for each neighbour $e$ of the origin $0$ a sequence of jobs
with bases $B_1,B_2, \ldots, B_R$ so that, for each $i$, 
$B_i \cap B_{i+1} \ne \varnothing$, $0 \in B_1$ and $e \in B_R$ and
the rate at which job with base $B_i$ arrives is strictly positive.
Taking $R_1$ to be the maximum over the $R$s as the neighbour $e$ 
varies and $c$ to
be the minimum over the rates $B_i$ as $e$ and $i$ vary, we obtain that,
for any $x$, there exists a
``path" $B_1,B_2, \ldots, B_R$ so that $R \le d R_1|x|_\infty$, 

for each $i$, $B_i \cap B_{i+1} \ne \varnothing$, 
$0 \in B_1$ and $x \in B_R$ and
the rate at which job $B_i$ arrives is at least $c$. Thus for every $x \in
[-c_1V,c_1V]^d $, the
probability that $W(x,V/2) \le V/2$ is bounded by the probability that a
parameter $Vc/2$ Poisson process is less than $dc_1VR_1$. 
The result now follows easily from Poisson tail probabilities.
\end{proof}

We note that if $V > 4M$ (assuming as we may that $\epsilon < 1/3$) then
$\quad W(x,t) > V/4$, for $t \in [V/2, 3V/4]$, implies that event $B^x_M $
does not occur.
This implies that 
\begin{prop} \label{lln}
If for some $x$ with $|x| \le KV$ we have $W(x,0) \ge V >4M$, then with
probability at least $c_1$,
\[
\frac{1}{(2Kn+1)^d} \sum_{|y|\le KN}\I_{B^y_M} \le 1- c_1/(2K)^d < 1-\epsilon,
\]
for $\epsilon$ fixed  small enough.
\end{prop}
This yields the simple corollary
\begin{cor} 
For $M$ and $\epsilon$ as above, let $A(V,K) $ be the event that
$W(x,0) > t$ for some $t > V$ and some $|x| \le Kt$.
Then, under measure $\nu$, the probability that $A(V,K)$ occurs is less 
than $\epsilon/c_1$ provided $c_1/(2K)^d > \epsilon$.
\end{cor}
From this result our claim is straightforward.


\begin{figure}[b]
\begin{center}
\epsfig{file=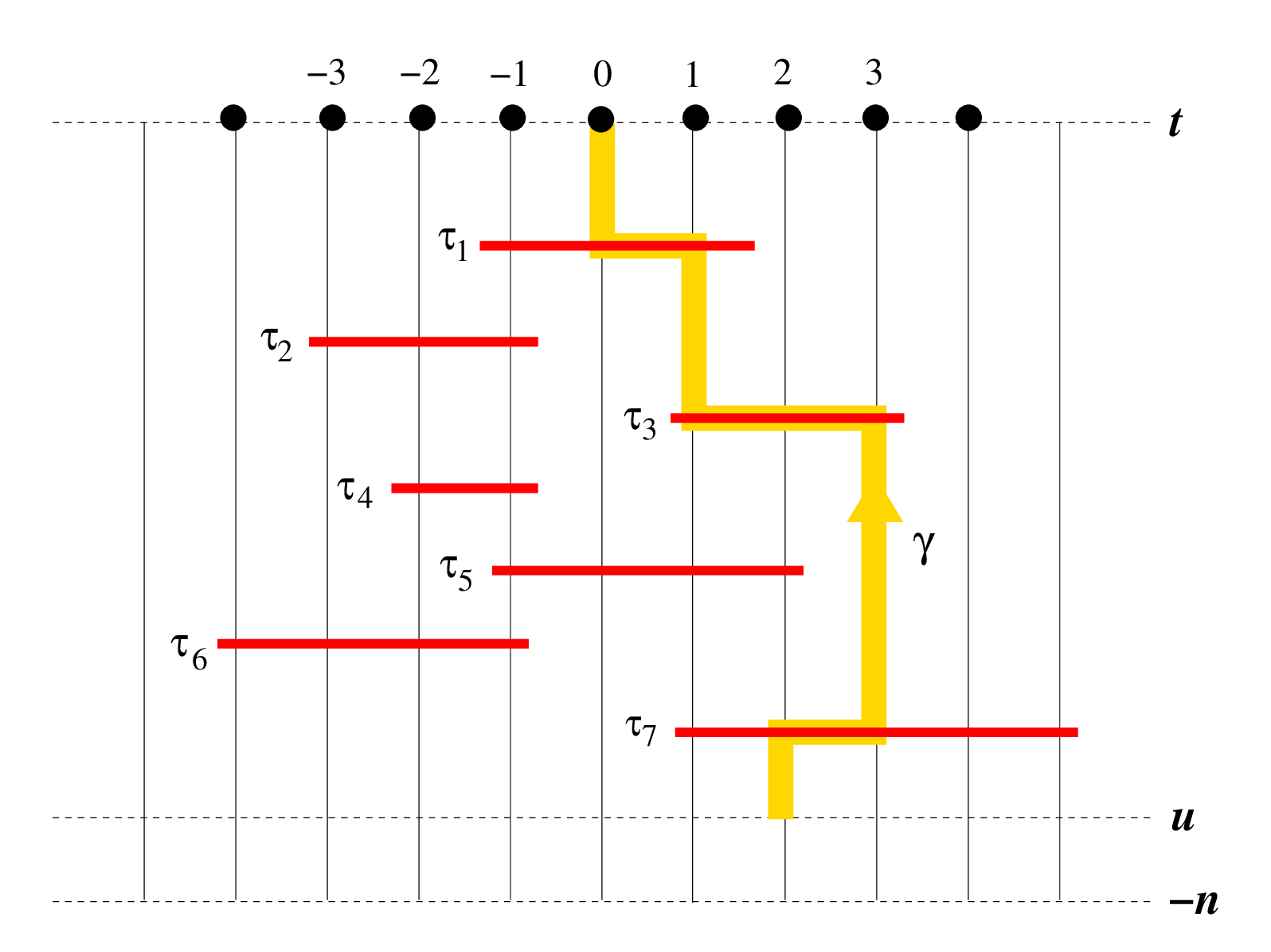,width=13cm}
\caption{\small \em Graphical representation of the part of
the process responsible for the computation of the profile $W(t,\cdot)$
when the system starts with $W(-n,\cdot)\equiv 0$. Consider the
evaluation of $W(t,0)$ at site $x=0$. Horizontal intervals
represent hailstone (job) arrivals with heights $\tau_i$.
Only those arrivals which can potentially influence
$W(t,0)$ are shown. Consider a path $\gamma$ as indicated,
from $(u,2)$ to $(t,0)$.
Its score is $V(\gamma)= \tau_1+\tau_3+\tau_7-(t-u)$.
$W(t,0)$ is the maximum of these scores over all such paths starting
from some $(u,y)$ and ending at $(t,0)$.}
\label{graphicalrep}
\end{center}
\end{figure}

\end{document}